\documentclass[a4paper,twoside,11pt]{article}
\usepackage{a4wide,graphicx,fancyhdr,amsmath,amssymb,mathtools,yfonts}
\usepackage[all]{xy}
\usepackage[utf8]{inputenc}
\usepackage{amsthm}
\usepackage[english]{babel}
\usepackage{chngcntr}
\usepackage{ifthen}
\usepackage{calc}
\usepackage{hyperref}

\newtheorem{theorem}{Theorem}
\newtheorem{lemma}[theorem]{Lemma}

\newtheorem{definition}[theorem]{Definition}

\title{\vspace{-\baselineskip}\sffamily\bfseries Unit equations and Fermat surfaces in positive characteristic}
\author{Peter Koymans, Carlo Pagano
\\{\tt p.h.koymans@math.leidenuniv.nl}
\\{\tt c.pagano@math.leidenuniv.nl}}

\date{\today}

\begin{document}
\maketitle

\begin{abstract}
In this article we study the three-variable unit equation $x + y + z = 1$ to be solved in $x, y, z \in \mathcal{O}_S^\ast$, where $\mathcal{O}_S^\ast$ is the $S$-unit group of some global function field. We give upper bounds for the height of solutions and the number of solutions. We also apply these techniques to study the Fermat surface $x^N + y^N + z^N = 1$.
\end{abstract}

\section{Introduction}
Let $K$ be a finitely generated field over $\mathbb{F}_p$ of transcendence degree $1$. Denote by $\mathbb{F}_q$ the algebraic closure of $\mathbb{F}_p$ inside $K$, which is a finite extension of $\mathbb{F}_p$. Let $M_K$ be the set of places of $K$ and let $S \subseteq M_K$ be a finite subset. To avoid degenerate cases, we will assume that $|S| \geq 2$ throughout the paper. We define $\omega(S) = \sum_{v \in S} \text{deg}(v)$ and we let $H_K$ be the usual height. For a precise definition of $\text{deg}(v)$ and $H_K$ we refer the reader to Section \ref{sPre}. Mason \cite{Mason} and Silverman \cite{Silverman2} independently considered the equation
\begin{align}
\label{eUnit2v}
x + y = 1 \text{ in } x, y \in \mathcal{O}_S^\ast.
\end{align}
If $x, y \not \in K^p$ is a solution to (\ref{eUnit2v}), they showed that
\begin{align}
\label{eHeight}
H_K(x) = H_K(y) \leq \omega(S) + 2g - 2,
\end{align}
where $g$ is the genus of $K$. Previously, Stothers \cite{Stothers} proved (\ref{eHeight}) for polynomials $x, y \in \mathbb{C}[t]$.

It is important to note that the condition $x, y \not \in K^p$ can not be removed. Indeed if we have a solution to (\ref{eUnit2v}), then we find that
\[
x^{p^k} + y^{p^k} = 1
\]
is also a solution to (\ref{eUnit2v}) for all integers $k \geq 0$ due to Frobenius, but the heights $H_K(x^{p^k})$ and $H_K(y^{p^k})$ become arbitrarily large. This new phenomenon is the main difficulty in dealing with two variable unit equations in positive characteristic.

The work of Mason and Silverman has been extended in various directions. Hsia and Wang \cite{HW} looked at the equation
\begin{align}
\label{eUnitnv}
x_1 + \cdots + x_n = 1 \text{ in } x_1, \ldots, x_n \in \mathcal{O}_S^\ast.
\end{align}
They were able to deduce a height bound similar to (\ref{eHeight}) under the condition that $x_1, \ldots, x_n$ are linearly independent over $K^p$. In particular it follows that under the same condition there are only finitely many solutions $x_1, \ldots, x_n$. Derksen and Masser \cite{DM} considered (\ref{eUnitnv}) without the restriction that $x_1, \ldots, x_n$ are linearly independent over $K^p$. In this case it is not a priori clear what the structure of the solution set should be, but Derksen and Masser give a completely explicit description that we repeat here in the special case that $n = 3$.

They define so-called one dimensional Frobenius families to be
\[
\mathcal{F}(\mathbf{u}) := \{(u_1, u_2, u_3)^{p^e} : e \geq 0\}
\]
for $\mathbf{u} = (u_1, u_2, u_3) \in (K^\ast)^3$ and two dimensional Frobenius families
\[
\mathcal{F}_a(\mathbf{u}, \mathbf{v}) := \left\{\left((u_1, u_2, u_3) (v_1, v_2, v_3)^{p^{af}}\right)^{p^e} : e, f \geq 0\right\}
\]
for $a \in \mathbb{Z}_{\geq 1}$, $\mathbf{u} = (u_1, u_2, u_3) \in (K^\ast)^3$, $\mathbf{v} = (v_1, v_2, v_3) \in (K^\ast)^3$, where all multiplications of tuples are taken coordinate-wise. Then Derksen and Masser prove that the solution set of
\begin{align}
\label{eUnit}
x + y + z = 1 \text{ in } x, y, z \in \mathcal{O}_S^\ast
\end{align}
is equal to a finite union of one dimensional and two dimensional Frobenius families. On top of that Derksen and Masser give effective height bounds for $\mathbf{u}$ and $\mathbf{v}$, which can be seen as another direct generalization of (\ref{eHeight}). In principle this also gives an upper bound on the total number of Frobenius families that one may need to describe the solution set of (\ref{eUnit}), but the resulting bounds are far from optimal. Leitner \cite{Leitner} computed the full solution set of (\ref{eUnit}) in the special case $S = \{0, 1, \infty\}$ and $K = \mathbb{F}_p(t)$.

In this paper we give explicit upper bounds for the height of $\mathbf{u}$ and $\mathbf{v}$ in the case $n = 3$. Together with a ``gap principle'' we will use this to give an upper bound on the number of Frobenius families. For the two variable unit equation $x + y = 1$ such upper bounds have already been established by Voloch \cite{Voloch} and by Koymans and Pagano \cite{KoymansPagano} using different methods than in this paper. The upper bound in the latter paper has the particularly pleasant feature that it does not depend on $p$. This paper is based on the paper of Beukers and Schlickewei \cite{BS}, who had previously established a finiteness result for the two variable unit equation in characteristic $0$.

Let $g$ and $\gamma$ be respectively the genus and the gonality of $K$. Put
\[
c_{K, S} := 2\omega(S) + 4g - 4 + 4\gamma, \ c'_{K, S} := 2c_{K, S} \cdot (\omega(S) + 4c_{K, S} + 2g - 2) + 3c_{K, S}.
\]
Define the following three sets
\begin{align*}
A := \{\mathbf{x} = (x, y, z) \in (\mathcal{O}_S^\ast)^3 : x + y + &z = 1, x, y, z \not \in \mathbb{F}_q^\ast, H_K(x), H_K(y), H_K(z) \leq c'_{K, S}\}, \\
B_1 := \{(\mathbf{u}, \mathbf{v}) \in (\mathcal{O}_S^\ast)^3 \times (\mathcal{O}_S^\ast)^3: &\mathbf{u}, \mathbf{v} \not \in (\mathbb{F}_q^\ast)^3, u_i \not \in \mathbb{F}_q^\ast \text{ or } v_i \not \in \mathbb{F}_q^\ast \text{ for } i = 1, 2, 3, \\
&H_K(u_i) \leq c_{K, S} \text{ for } i = 1, 2, 3, \\
&H_K(v_i) \leq \omega(S) + 2g - 2 \text{ for } i = 1, 2, 3, \\
&u_1 v_1^{p^{f}} + u_2 v_2^{p^{f}} + u_3 v_3 ^{p^{f}} = 1 \text{ for all } f \in \mathbb{Z}_{\geq 0}\}, \\
B_q := \{(\mathbf{u}, \mathbf{v}) \in (\mathcal{O}_S^\ast)^3 \times (\mathcal{O}_S^\ast)^3: &\mathbf{u}, \mathbf{v} \not \in (\mathbb{F}_q^\ast)^3, u_i \not \in \mathbb{F}_q^\ast \text{ or } v_i \not \in \mathbb{F}_q^\ast \text{ for } i = 1, 2, 3, \\
&H_K(u_i) \leq c_{K, S}, \text{ for } i = 1, 2, 3, \\
&H_K(v_i) \leq \frac{q}{p}(\omega(S) + 2g - 2), \text{ for } i = 1, 2, 3, \\
&u_1 v_1^{q^{f}} + u_2 v_2^{q^{f}} + u_3 v_3^{q^{f}} = 1 \text{ for all } f \in \mathbb{Z}_{\geq 0}\}.
\end{align*}

\begin{theorem}
\label{tHeight}
For all $x, y, z \not \in \mathbb{F}_q$ we have the following equivalence: $x, y, z$ is a solution to (\ref{eUnit}) if and only if $(x, y, z)$ is an element of one of the following three sets
\begin{equation}
\label{eUn}
\bigcup_{\mathbf{x} \in A} \mathcal{F}(\mathbf{x}), \bigcup_{(\mathbf{u}, \mathbf{v}) \in B_1} \mathcal{F}_1(\mathbf{u}, \mathbf{v}), \bigcup_{(\mathbf{u}, \mathbf{v}) \in B_q} \mathcal{F}_{\log_p(q)}(\mathbf{u}).
\end{equation}
\end{theorem}

\begin{theorem}
\label{tCount}
There are a subset $C_1$ of $(K^\ast)^3$ and subsets $C_2$ and $C_3$ of $(K^\ast)^3 \times (K^\ast)^3$ with the following properties
\begin{itemize}
\item $|C_1| \leq 93q^2 \cdot (\log_{\frac{5}{4}}(3c'_{K, S}) + 1)^2 \cdot (15 \cdot 10^6)^{|S|}$;
\item $|C_2| \leq 961 \cdot p^5 \cdot 19^{4|S|}$;
\item $|C_3| \leq 961 \cdot \log_p(q) \cdot q^5 \cdot 19^{4|S|}$;
\item for all $x, y, z \not \in \mathbb{F}_q$ we have the following equivalence: $x, y, z$ is a solution to (\ref{eUnit}) if and only if $(x, y, z)$ is an element of one of the following three sets
\[
\bigcup_{\mathbf{x} \in C_1} \mathcal{F}(\mathbf{x}), \bigcup_{(\mathbf{u}, \mathbf{v}) \in C_2} \mathcal{F}_1(\mathbf{u}, \mathbf{v}), \bigcup_{(\mathbf{u}, \mathbf{v}) \in C_3} \mathcal{F}_{\log_p(q)}(\mathbf{u}, \mathbf{v}).
\]
\end{itemize}
\end{theorem}

Let $N > 0$ be an integer. As is well known there is a strong relation between unit equations and the Fermat equation
\[
x_1^N + \ldots + x_m^N = 1
\]
to be solved in $x_1, \ldots, x_m \in k(t)$ for some field $k$. This relation has been used in characteristic $0$ by for example Voloch \cite{Voloch2} and Bombieri and Mueller \cite{BM}. However, it is not clear how these methods can be made to work in characteristic $p > 0$. For example it would be natural to try and use a height bound for (\ref{eUnitnv}), but this is only possible when $x_1^N, \ldots, x_m^N$ are linearly independent over $K^p$. In the special case $m = 2$ this problem has been considered by Silverman \cite{Silverman}, but unfortunately his main theorem is false. A correct statement with proof can be found in \cite{Koymans}. Here we will analyze the case $m = 3$.

\begin{definition}
\label{dGood}
We say that an integer $N > 0$ is $(x, p)$-good if the congruence
\[
a p^s + b \equiv 0 \mod N
\]
has no solutions in integers $s \geq 0$, $0 < a, b \leq x$.
\end{definition}

We remark that for a given tuple $(x, p)$ a positive density of the primes is $(x, p)$-good. Indeed, if $N > 2$ is a prime satisfying
\[
\left(\frac{-1}{N}\right) = -1, \quad \left(\frac{p}{N}\right) = 1, \quad \left(\frac{a}{N}\right) = 1 \text{ for } 0 < a \leq x,
\]
then $N$ is $(x, p)$-good.

\begin{theorem}
\label{tFermat}
Let $p > 480$ be a prime number and suppose that $N$ is a $(480, p)$-good integer.  If we further suppose that $\gcd(N, p) = 1$, then the Fermat surface
\begin{align}
\label{eFermat}
x^N + y^N + z^N = 1
\end{align}
has no solutions $x, y, z \in \mathbb{F}_p(t)$ satisfying $x, y, z \not \in \mathbb{F}_p(t^p)$ and $x/y, x/z, y/z \not \in \mathbb{F}_p(t^p)$.
\end{theorem}

Note that Theorem \ref{tFermat} is in stark contrast with the behavior of the Fermat surface in characteristic $0$ \cite{Voloch2}. Remarkably enough it turns out that Theorem \ref{tFermat} becomes false if we drop any of the last two conditions, see Section \ref{sCounter}. We will also explain there why we need the condition that $N$ is $(480, p)$-good. The rough reason is that if $N$ is not $(1, p)$-good, then the Fermat surface is known to be unirational \cite{Ronald}. Our work shows that the unirationality of these surfaces is strongly related to the two-dimensional Frobenius families appearing in Theorem \ref{tHeight}. For precise details, we refer the reader to Section \ref{sCounter}.

\section{Preliminaries}
\label{sPre}
In this section we start by defining heights, which will play a key role throughout the paper. Furthermore, we give two important lemmata about heights.

\subsection{Definition of height}
Recall that $K$ is a finitely generated field over $\mathbb{F}_p$ of transcendence degree $1$ and that $\mathbb{F}_q$ is the algebraic closure of $\mathbb{F}_p$ inside $K$. We further recall that $M_K$ is the set of places of $K$. The valuation ring of a place $v \in M_K$ is given by
\[
O_v := \{x \in K : v(x) \geq 0\}.
\]
This is a discrete valuation ring with maximal ideal $m_v := \{x \in K : v(x) > 0\}$. The residue class field $O_v/m_v$ naturally becomes a finite field extension of $\mathbb{F}_q$. Hence
\[
\deg(v) := [O_v/m_v : \mathbb{F}_q]
\]
is a well-defined integer. With these definitions it turns out that the sum formula holds for all $x \in K^\ast$, i.e.
\[
\sum_v v(x) \deg(v) = 0,
\]
where here and below $\sum_v$ denotes a summation over $v \in M_K$. This allows us to define the height for $x \not \in \mathbb{F}_q$ as follows
\[
H_K(x) := [K : \mathbb{F}_q(x)] = \sum_{v \in M_K} \max(v(x) , 0) \deg(v) = \sum_{v \in M_K} -\min(v(x), 0) \deg(v).
\]
For $x \in \mathbb{F}_q$ we set $H_K(x) := 0$. More generally, we define the projective height to be
\[
H_K(x_0 : \ldots : x_n) := -\sum_{v \in M_K} \min(v(x_0), \ldots, v(x_n)) \deg(v)
\]
for $(x_0 : \ldots : x_n) \in \mathbb{P}^n(K)$, which is well-defined due to the sum formula. One can recover the usual height by the identity $H_K(x) = H_K(1 : x)$.

\subsection{Height lemmata}
Pick $t \in K^\ast$ such that $K/\mathbb{F}_q(t)$ is of the minimal possible degree $\gamma$, the gonality of $K$. Then it follows that $K/\mathbb{F}_q(t)$ is a separable extension. Let $D$ be the extension to $K$ of the derivation $\frac{d}{dt}$ on $\mathbb{F}_q(t)$. Then $\text{ker}(D) = K^p$. We will fix such a derivation $D$ for the remainder of the paper. We let $H_K$ be the height as just defined. For $x \in K^\ast$ we write $\omega(x) = \sum_{v: v(x) \neq 0} \text{deg}(v)$. 

\begin{lemma}
\label{lHeight}
Let $f \in K^\ast$. Then for $f \not \in K^p$
\[
H_K\left(\frac{Df}{f}\right) \leq \omega(f) + 2g - 2 + 2\gamma,
\]
where $g$ is the genus of $K$.
\end{lemma}

\begin{proof}
We have
\[
H_K\left(\frac{Df}{f}\right) = \frac{1}{2} \sum_{v} \left|v\left(\frac{Df}{f}\right)\right| \text{deg}(v).
\]
For a valuation $v$ of $K$, denote by $w(v):=v_{|\mathbb{F}_q(t)}$ the valuation lying below $v$ in $\mathbb{F}_q(t)$. Denote by $z_v$ a choice of a uniformizer at $v$ and similarly, denote by $z_{w(v)}$ a choice of a uniformizer at $w(v)$. Then
\[
v\left(\frac{Df}{f}\right) = v\left(\frac{df}{dz_v}\right) - v(f) - v\left(\frac{dz_{w(v)}}{dz_{v}}\right) - v\left(\frac{dt}{dz_{w(v)}}\right).
\]
Therefore we get that
\begin{gather*}
H_K\left(\frac{Df}{f}\right) = \frac{1}{2} \sum_v \left|v\left(\frac{Df}{f}\right)\right| \text{deg}(v) \leq \\
\frac{1}{2} \cdot \left(\sum_v \left|v\left(\frac{df}{dz_v}\right) - v(f)\right| \text{deg}(v) + \sum_v \left|v\left(\frac{dt}{dz_{w(v)}}\right)\right| \text{deg}(v) + \sum_v \left|v\left(\frac{dz_{w(v)}}{dz_{v}}\right)\right| \text{deg}(v) \right).
\end{gather*}
We call the three inner sums respectively $T_1, T_2, T_3$.

\vspace{\baselineskip}

\noindent \textbf{Bound for $T_1$} \\
By the Riemann-Roch Theorem, see e.g. equation (5) of page 96, chapter 6 in \cite{Mason}, we have for $f \not \in K^p$ that
\begin{align}
\label{eGen}
\sum_v v\left(\frac{df}{dz_v}\right) \text{deg}(v) = 2g - 2
\end{align}
and hence by the sum formula
\[
\sum_{v} \left(v\left(\frac{df}{dz_v}\right) - v(f)\right) \text{deg}(v) = 2g - 2.
\]
Furthermore $v\left(\frac{df}{dz_v}\right) - v(f) < 0$ implies $v\left(\frac{df}{dz_v}\right) - v(f) = -1$. Therefore 
\[
\sum_{v: v\left(\frac{df}{dz_v}\right) <v(f)} \left|v\left(\frac{df}{dz_v}\right) - v(f)\right| \text{deg}(v) \leq \omega(f)
\]
and thus 
\[
\sum_{v: v\left(\frac{df}{dz_v}\right) \geq v(f)} \left(v\left(\frac{df}{dz_v}\right) - v(f)\right) \text{deg}(v) \leq 2g - 2 + \omega(f).
\]
In total we get that 
\[
T_1 \leq 2\omega(f) + 2g - 2.
\]

\noindent \textbf{Bound for $T_2$} \\
Using (\ref{eGen}) over $\mathbb{F}_q(t)$, one immediately gets the bound
\[
T_2 \leq 2\gamma.
\]

\noindent \textbf{Bound for $T_3$} \\
Denote by $K_v/\mathbb{F}_{q}((z_{w(v)}))$ the extension of local fields, by $e(v/w(v))$ the ramification degree and recall that the residue degree is just $\text{deg}(v)$. Hence we have the relation
\[
e(v/w(v)) \cdot \text{deg}(v) = [K_v : \mathbb{F}_q((z_{w(v)}))].
\]
We find that $K_v/\mathbb{F}_{q^{\text{deg}(v)}}((z_{w(v)}))$ is totally ramified, and therefore given by a degree $e(v/w(v))$ Eisenstein polynomial, say
\[
p(x) := x^{e(v/w(v))} + \sum_{i=0}^{e(v/w(v)) - 1} a_ix^i.
\]
We can choose $p(x)$ in such a way that $p(z_v) = 0$ and $p(0) = -z_{w(v)}$. Let $p'(x)$ be the formal derivative of $p$ with respect to $x$. From the identity $p(z_v) = 0$ we get after applying $\frac{d}{dz_v}$
\[
p'(z_v) = \frac{dz_{w(v)}}{dz_v} - \sum_{i=1}^{e(v/w(v))-1} \frac{da_i}{dz_v} z_v^i.
\]
On the other hand, by virtue of $p(x)$ being Eisenstein, we get that
\[
v\left(\frac{dz_{w(v)}}{dz_v}\right) < v\left(z_{v}^i\frac{da_i}{dz_v}\right)
\]
for every $i \in \{1, \ldots, e(v/w(v)) -1\}$. Therefore we deduce by the non-archimedean strong triangle inequality that
\[
v\left(\frac{dz_{w(v)}}{dz_v} - \sum_{i=1}^{e(v/w(v))-1}z_{v}^i\frac{da_i}{dz_v}\right) = v\left(\frac{dz_{w(v)}}{dz_v}\right)
\]
and thus
\[
v\left(\frac{dz_{w(v)}}{dz_v}\right) = v(p'(z_v)).
\]
By chapter 3, section 6 of \cite{Serre}, we have that $v(p'(z_v))$ is what Stichtenoth \cite{Stichtenoth} calls the \emph{different exponent} $d(v|w(v))$. Therefore we deduce that 
\[
\sum_v \left|v\left(\frac{dz_{w(v)}}{dz_v}\right)\right| \text{deg}(v) = \text{deg}(\text{Diff}(K/\mathbb{F}_q(t))
\]
where $\text{Diff}(-/-)$ denotes the \emph{different divisor}, i.e. the sum of all valuations of $K$ weighted with their different exponent. Thus by Corollary 3.4.14 in \cite{Stichtenoth}
\[
\sum_v \left|v\left(\frac{dz_{w(v)}}{dz_v}\right)\right| \text{deg}(v) = 2g - 2 + 2[K : \mathbb{F}_q(t)].
\]
Altogether we have obtained a bound
\[
T_3 \leq 2g - 2 + 2\gamma.
\]

\noindent \textbf{Conclusion of proof} \\
In total we get
\[
H_K\left(\frac{Df}{f}\right) \leq \frac{1}{2}(T_1 + T_2 + T_3) \leq \omega(f) + 2g - 2 + 2\gamma,
\]
which is the desired inequality.
\end{proof}

\noindent We will repeatedly use the following two theorems.

\begin{theorem}
\label{tMason}
Let $x, y \in \mathcal{O}_S^\ast$. If $x, y \not \in K^p$ and
\[
x + y = 1,
\]
then we have
\[
H_K(x) = H_K(y) \leq \omega(S) + 2g - 2.
\]
\end{theorem}

\begin{proof}
See \cite{Mason} and \cite{Silverman2}.
\end{proof}

\begin{theorem}
\label{tKP}
Let $K$ be a field of characteristic $p > 0$ and let $G$ be a finitely generated subgroup of $K^\ast \times K^\ast$ of rank $r$. Then the equation
\[
x + y = 1 \text{ in } (x, y) \in G
\]
has at most $31 \cdot 19^r$ solutions $(x, y)$ satisfying $(x, y) \not \in G^p$.
\end{theorem}

\begin{proof}
This is Theorem 2 of \cite{KoymansPagano}.
\end{proof}

\section{Proof of Theorem \ref{tHeight}}
\begin{proof}
By construction $\mathcal{F}(\mathbf{x})$ is a solution to (\ref{eUnit}) for $\mathbf{x} \in A$ and likewise all elements of $\mathcal{F}_a(\mathbf{u}, \mathbf{v})$ are solutions to (\ref{eUnit}). Hence it suffices to prove the only if part of Theorem \ref{tHeight}. Let $x, y, z$ be a solution of (\ref{eUnit}) with $x, y, z \not \in \mathbb{F}_q$. Note that the sets as given in equation (\ref{eUn}) are all invariant under taking $p$-th roots. Since $x, y, z \not \in \mathbb{F}_q$, we can keep taking $p$-th roots of the tuple $(x, y, z)$ until $x$, $y$ or $z$ is not in $K^p$. For ease of notation we will keep using the same letters for the new $x$, $y$ and $z$. By symmetry we may assume that $z \not \in K^p$. Then also $x \not \in K^p$ or $y \not \in K^p$. Again we may assume by symmetry that $y \not \in K^p$. Now we distinguish two cases.

\vspace{\baselineskip}

\noindent \textbf{Case I:}
First suppose that $x \in K^p$. Then using
\[
x + y + z = 1
\]
we find after differentiating with respect to $D$
\[
\frac{Dy}{y} y + \frac{Dz}{z} z = 0.
\]
We can rewrite this as follows
\begin{align*}
x + y\left(1 - \frac{z}{Dz} \frac{Dy}{y} \right) &= 1 \\
x + z\left(1 - \frac{y}{Dy} \frac{Dz}{z} \right) &= 1.
\end{align*}
Define $a_2 := 1 - \frac{z}{Dz} \frac{Dy}{y}$ and $b_3 := 1 - \frac{y}{Dy} \frac{Dz}{z}$. Note that $a_2 = 0$ implies $x = 1$, contrary to our assumption $x \not \in \mathbb{F}_q$. Similarly $b_3 \neq 0$. The above system of equations implies that either $b_3,a_2 \not \in \mathcal{O}_S^\ast$ or  $b_3,a_2 \in \mathcal{O}_S^\ast$. Consider first the case $b_3,a_2 \not \in \mathcal{O}_S^\ast$. By Lemma \ref{lHeight} we have
\[
H_K(b_3) \leq c_{K, S}.
\]
Hence $b_3 z \not \in K^{p^l}$, where $l := \lfloor \log_p {c_{K, S}} \rfloor + 1$. Write $x = \delta^{p^s}$ and $b_3z = \epsilon^{p^s}$, with $\delta, \epsilon \not \in K^p$. Note that $\delta + \epsilon = 1$, so an application of Theorem \ref{tMason} gives
\[
H_K(\delta) = H_K(\epsilon) \leq \omega(S) + 2c_{K, S} + 2g - 2,
\]
where we used that $\omega(b_3) \leq 2H_K(b_3) \leq 2c_{K, S}$. We conclude that
\[
H_K(x) = H_K(b_3 z) = p^s H_K(\delta) = p^s H_K(\epsilon) \leq c_{K, S} \cdot (\omega(S) + 2c_{K, S} + 2g - 2),
\] 
since $p^s \leq p^{l - 1} \leq c_{K, S}$. \\
We now consider the case that $a_2, b_3 \in \mathcal{O}_S^\ast$. Since $x \not \in \mathbb{F}_q$ there is $x' \not \in K^p$ such that $x = x'^{p^s}$ for some $s > 0$. There are also $y', z' \in \mathcal{O}_S^\ast$ such that
\begin{align*}
x' + a_2 y' &= 1 \\
x' + b_3 z' &= 1.
\end{align*}
Applying Theorem \ref{tMason} again yields
\[
H_K(x') = H_K(a_2 y') \leq \omega(S) + 2g - 2.
\]
We conclude that
\[
(x, y, z) \in \mathcal{F}_1((1, a_2^{-1}, b_3^{-1}), (x', a_2y', b_3z')),
\]
with $a_2,b_3 \not \in \mathbb{F}_q$, since otherwise $y,z \in K^p$, which would be a contradiction.

\vspace{\baselineskip}

\noindent \textbf{Case II:}
Now suppose $x \not \in K^p$. We start by dealing with the case $\frac{x}{Dx} \neq \frac{y}{Dy}$, $\frac{x}{Dx} \neq \frac{z}{Dz}$, $\frac{y}{Dy} \neq \frac{z}{Dz}$. Then we find that
\[
x + y + z = 1
\]
and after differentiating with respect to $D$
\[
\frac{Dx}{x} x + \frac{Dy}{y} y + \frac{Dz}{z} z = 0.
\]
This is equivalent to
\begin{align*}
x\left(1 - \frac{z}{Dz} \frac{Dx}{x} \right) + y\left(1 - \frac{z}{Dz} \frac{Dy}{y} \right) &= 1 \\
x\left(1 - \frac{y}{Dy} \frac{Dx}{x} \right) + z\left(1 - \frac{y}{Dy} \frac{Dz}{z} \right) &= 1.
\end{align*}
For convenience we define
\[
a_1 := 1 - \frac{z}{Dz} \frac{Dx}{x}, a_2 := 1 - \frac{z}{Dz} \frac{Dy}{y}, b_1 := 1 - \frac{y}{Dy} \frac{Dx}{x}, b_3 := 1 - \frac{y}{Dy} \frac{Dz}{z}.
\]
By our assumption we know that the coefficients $a_1$, $a_2$, $b_1$ and $b_3$ are not zero. If one of the coefficients, say $a_1$, does not lie in $\mathcal{O}_S^\ast$, we can proceed exactly as before obtaining the bound
\[
H_K(a_1 x) = H_K(a_2 y) \leq c_{K, S} \cdot (\omega(S) + 4c_{K, S} + 2g - 2).
\]
So now suppose that $a_1, a_2, b_1, b_3 \in \mathcal{O}_S^\ast$, but also suppose that $d := \frac{a_1}{b_1} \not \in \mathbb{F}_q^\ast$. In this case we have
\[
H_K(d) \leq 2c_{K, S}
\]
and therefore $a_1 x \not \in K^{p^l}$ or $b_1 x \not \in K^{p^l}$ with $l := \lfloor \log_p {2c_{K, S}} \rfloor + 1$. Suppose that $a_1 x \not \in K^{p^l}$. Then Theorem \ref{tMason} gives
\[
H_K(a_1 x) = H_K(a_2 y) \leq 2c_{K, S} \cdot (\omega(S) + 4c_{K, S} + 2g - 2)
\]
and the other case can be dealt with in exactly the same way.

Finally suppose that $a_1, a_2, b_1, b_3 \in \mathcal{O}_S^\ast$ and $d \in \mathbb{F}_q^\ast$. If we additionally suppose that one of the coefficients is in $\mathbb{F}_q^\ast$, another application of Theorem \ref{tMason} yields
\[
H_K(a_1 x) = H_K(a_2 y) = H_K(b_1 x) = H_K(b_3 z) \leq \omega(S) + 2g - 2.
\]
Hence we will assume that $a_1, a_2, b_1, b_3 \not \in \mathbb{F}_q^\ast$ from now on. If $a_1 x \in \mathbb{F}_q^\ast$, we immediately get a height bound for $x$. So we may further assume that $a_1 x \not \in \mathbb{F}_q^\ast$. Then let $l \geq 0$ be the largest integer such that $a_1 x \in K^{q^l}$. Define $x' \in \mathcal{O}_S^\ast$ as
\[
(a_1 x')^{q^l} = a_1 x
\]
and then define $y', z' \in \mathcal{O}_S^\ast$ such that
\begin{align*}
a_1 x' + a_2 y' &= 1 \\
b_1 x' + b_3 z' &= 1.
\end{align*}
Furthermore,
\[
H_K(a_1 x') = H_K(a_2 y') \leq \frac{q}{p} (\omega(S) + 2g - 2)
\]
and
\[
(x, y, z) \in \mathcal{F}_{\log_p(q)}((a_1^{-1}, a_2^{-1}, b_3^{-1}), (a_1x', a_2y', b_3z')).
\]
This deals with the case $x \not \in K^p$ and $\frac{x}{Dx} \neq \frac{y}{Dy}$, $\frac{x}{Dx} \neq \frac{z}{Dz}$, $\frac{y}{Dy} \neq \frac{z}{Dz}$.

We still have to deal with the case $x \not \in K^p$ and $\frac{x}{Dx} = \frac{y}{Dy}$ or $\frac{x}{Dx} = \frac{z}{Dz}$ or $\frac{y}{Dy} = \frac{z}{Dz}$. Recall that $y, z \not \in K^p$ as well, hence the three cases are symmetrical. So we will only deal with the case $\frac{y}{Dy} = \frac{z}{Dz}$. Then we get the equations
\[
x \left(1 - \frac{y}{Dy} \frac{Dx}{x} \right) = x \left(1 - \frac{z}{Dz} \frac{Dx}{x} \right) = 1
\]
and hence
\[
H_K(x) \leq c_{K, S}.
\]
Our equation implies that $a_1 := b_1 := 1 - \frac{y}{Dy}\frac{Dx}{x} \in \mathcal{O}_S^\ast$. Substitution in the original equation yields
\[
\frac{1}{a_1} + y + z = 1
\]
or equivalently
\[
y + z = 1 - \frac{1}{a_1} = \frac{a_1 - 1}{a_1}.
\]
After putting $\alpha := \frac{a_1}{a_1 - 1}$ we get
\[
\alpha y + \alpha z = 1.
\]
Note that
\[
H_K(\alpha) = H_K(a_1) = H_K(x) \leq c_{K, S}.
\]
Suppose that $\alpha \not \in \mathcal{O}_S^\ast$. Just as before we find that $\alpha y \not \in K^{p^l}$, where $l := \lfloor \log_p {c_{K, S}} \rfloor + 1$. Then Theorem \ref{tMason} gives
\[
H_K(\alpha y) = H_K(\alpha z) \leq c_{K, S} \cdot (\omega(S) + c_{K, S} + 2g - 2).
\]
The last case is $\alpha \in \mathcal{O}_S^\ast$. Suppose that $\alpha \in \mathbb{F}_q^\ast$. From Theorem \ref{tMason} we deduce that
\[
H_K(\alpha y) = H_K(\alpha z) \leq \omega(S) + 2g - 2.
\]
So from now on we further assume that $\alpha \not \in \mathbb{F}_q^\ast$. If $\alpha y \in \mathbb{F}_q^\ast$ or $\alpha z \in \mathbb{F}_q^\ast$, we immediately get a height bound for respectively $y$ or $z$. So suppose that $\alpha y \not \in \mathbb{F}_q^\ast$ and $\alpha z \not \in \mathbb{F}_q$. Then there are $y', z' \not \in K^p$ and $s \in \mathbb{Z}_{\geq 0}$ such that $y'^{p^s} = \alpha y$ and $z'^{p^s} = \alpha z$ and we get an equation
\[
y' + z' = 1.
\]
Applying Theorem \ref{tMason} once more
\[
H_K(y') = H_K(z') \leq \omega(S) + 2g - 2.
\]
We conclude that
\[
(x, y, z) \in \mathcal{F}_1((x, \alpha^{-1}, \alpha^{-1}), (1, y', z')).
\]
This completes the proof.
\end{proof}

\section{Proof of Theorem \ref{tCount}} \label{2-Frob}
Define the set $B_1'$ by
\begin{align*}
B_1' := \{(\mathbf{u}, \mathbf{v}) &\in (\mathcal{O}_S^\ast)^3 \times (\mathcal{O}_S^\ast)^3: \mathbf{u}, \mathbf{v} \not \in (K^p)^3, u_i \not \in \mathbb{F}_q^\ast \text{ or } v_i \not \in \mathbb{F}_q^\ast, H_K(u_i) \leq c_{K, S}, \\
H_K(v_i) &\leq \omega(S) + 2g - 2, u_1 v_1^{p^{f}} + u_2 v_2^{p^{f}} + u_3 v_3^{p^{f}} = 1 \text{ for all } f \in \mathbb{Z}_{\geq 0}\}.
\end{align*}
For the reader's convenience we recall that in the definition of $B_1$ we only required that $\mathbf{u}, \mathbf{v} \not \in (\mathbb{F}_q^\ast)^3$ instead of the stronger condition $\mathbf{u}, \mathbf{v} \not \in (K^p)^3$. Nevertheless we have the equality
\[
\bigcup_{(\mathbf{u}, \mathbf{v}) \in B_1} \mathcal{F}_1(\mathbf{u}, \mathbf{v}) = \bigcup_{(\mathbf{u}, \mathbf{v}) \in B_1'} \mathcal{F}_1(\mathbf{u}, \mathbf{v}),
\]
so our goal will be to give an upper bound for the cardinality of $B_1'$. So suppose that $(\mathbf{u}, \mathbf{v}) \in B_1'$. Then we know that
\[
u_1v_1^{p^f} + u_2v_2^{p^f} + u_3v_3^{p^f} = 1
\]
for all $f \in \mathbb{Z}_{\geq 0}$. In fact, we will only use this equality for $f = 0, \ldots, 3$. Define
\[
A :=
\begin{pmatrix}
v_1 & v_2 & v_3 \\
v_1^{p} & v_2^{p} & v_3^{p} \\
v_1^{p^{2}} & v_2^{p^{2}} & v_3^{p^{2}}
\end{pmatrix}
.
\]
Our first goal is to show that $v_1, v_2, v_3$ are linearly dependent over $\mathbb{F}_p$. If not, then it would follow that $A$ is invertible. But we know that
\[
A
\begin{pmatrix}
u_1 \\
u_2 \\
u_3 \\
\end{pmatrix}
=
\begin{pmatrix}
1 \\
1 \\
1 \\
\end{pmatrix}
, \quad
A
\begin{pmatrix}
u_1^p \\
u_2^p \\
u_3^p \\
\end{pmatrix}
=
\begin{pmatrix}
1 \\
1 \\
1 \\
\end{pmatrix}
.
\]
This would imply that $\mathbf{u} \in (\mathbb{F}_p^\ast)^3$, contrary to our assumption $(\mathbf{u}, \mathbf{v}) \in B_1'$.

We conclude that $v_1, v_2, v_3$ are indeed linearly dependent over $\mathbb{F}_p$. Suppose that
\[
\alpha_1v_1 + \alpha_2v_2 + \alpha_3v_3 = 0
\]
with $\alpha_i \in \mathbb{F}_p$ not all zero. By symmetry we may suppose that $\alpha_3 \neq 0$. This yields
\begin{align}
\label{eDoubleFrob}
\left(u_1 - \frac{\alpha_1}{\alpha_3}u_3\right)v_1^{p^f} + \left(u_2 - \frac{\alpha_2}{\alpha_3}u_3\right)v_2^{p^f} = 1,
\end{align}
again for all $f \in \mathbb{Z}_{\geq 0}$. We will now suppose that $v_1, v_2$ are linearly dependent over $\mathbb{F}_p$ and derive a contradiction. If $\beta_1 v_1 = v_2$ for some $\beta_1 \in \mathbb{F}_p^\ast$, we find that
\[
\left(u_1 - \frac{\alpha_1}{\alpha_3}u_3\right)v_1^{p^f} + \beta \left(u_2 - \frac{\alpha_2}{\alpha_3}u_3\right)v_1^{p^f} = 1
\]
for all $f \in \mathbb{Z}_{\geq 0}$. Using this for $f = 0$ and $f = 1$ we conclude that $v_1 = v_1^p$, i.e. $v_1 \in \mathbb{F}_p^\ast$. This implies that also $v_2, v_3 \in \mathbb{F}_p^\ast$, contrary to our assumption $(\mathbf{u}, \mathbf{v}) \in B_1'$.

Hence we may assume that $v_1$ and $v_2$ are linearly independent over $\mathbb{F}_p$. From (\ref{eDoubleFrob}) we deduce that
\[
\lambda_1 := u_1 - \frac{\alpha_1}{\alpha_3}u_3 \in \mathbb{F}_p, \quad \lambda_2 := u_2 - \frac{\alpha_2}{\alpha_3}u_3 \in \mathbb{F}_p
\]
and therefore $\lambda_1v_1 + \lambda_2v_2 = 1$. We claim that at most one of $\alpha_1, \alpha_2, \lambda_1, \lambda_2$ is equal to zero.

It is clear that $\alpha_1$ and $\alpha_2$ can not be simultaneously equal to zero, and the same holds for $\lambda_1$ and $\lambda_2$. If $\alpha_1 = \lambda_1 = 0$, we find that $u_1 = 0$, which contradicts $u_1 \in \mathcal{O}_S^\ast$. Now suppose that $\alpha_1 = \lambda_2 = 0$. In this case we deduce that $u_1, v_1 \in \mathbb{F}_p^\ast$, again contrary to our assumption $(\mathbf{u}, \mathbf{v}) \in B_1'$. The remaining two cases can be dealt with symmetrically, establishing our claim.

Let us first suppose that $\alpha_1, \alpha_2, \alpha_3, \lambda_1, \lambda_2$ are all fixed and non-zero. Then we view the equations
\[
\lambda_1 = u_1 - \frac{\alpha_1}{\alpha_3}u_3, \quad \lambda_2 = u_2 - \frac{\alpha_2}{\alpha_3}u_3, \quad \lambda_1v_1 + \lambda_2v_2 = 1
\]
as unit equations to be solved in $u_1, u_2, u_3, v_1, v_2$. If one of the $u_i$ is in $K^p$, then it turns out that all the $u_i$ are in $K^p$, contradicting our assumption $\mathbf{u} \not \in (K^p)^3$. Henceforth we may assume that $u_1, u_2, u_3 \not \in K^p$ and similarly $v_1, v_2 \not \in K^p$. Theorem \ref{tKP} implies that there are at most $31 \cdot 19^{2|S|}$ solutions $(u_1, u_3)$ to $\lambda_1 = u_1 - \frac{\alpha_1}{\alpha_3}u_3$ and at most $31 \cdot 19^{2|S|}$ solutions $(v_1, v_2)$ to $\lambda_1v_1 + \lambda_2v_2 = 1$. Note that $u_1$ and $u_3$ determine $u_2$ and similarly $v_1$ and $v_2$ determine $v_3$. Hence there are at most $961 \cdot 19^{4|S|}$ possibilities for $(\mathbf{u}, \mathbf{v})$.

We will now treat the case $\lambda_2 = 0$ and $\alpha_1, \alpha_2, \alpha_3, \lambda_1$ fixed and non-zero. In this case we can treat the unit equation
\[
\lambda_1 = u_1 - \frac{\alpha_1}{\alpha_3}u_3
\]
exactly as before; it has at most $31 \cdot 19^{2|S|}$ solutions $(u_1, u_3)$. Using that $0 = \lambda_2 = u_2 - \frac{\alpha_2}{\alpha_3}u_3$, we see that $u_2$ is determined by $u_1$ and $u_3$. Note that $\lambda_2 = 0$ implies $\lambda_1 v_1 = 1$, i.e. $v_1 = \frac{1}{\lambda_1}$. We recall that
\[
\alpha_1 v_1 + \alpha_2 v_2 + \alpha_3 v_3 = 0
\]
and therefore
\[
\alpha_2 v_2 + \alpha_3 v_3 = -\frac{\alpha_1}{\lambda_1}.
\]
If $v_2 \in K^p$, then also $v_3 \in K^p$ and we conclude that $(v_1, v_2, v_3) \in (K^p)^3$. This is again a contradiction, so suppose that $v_2, v_3 \not \in K^p$. We are now in the position to apply Theorem \ref{tKP}, which shows that there are at most $31 \cdot 19^{2|S|}$ solutions $(v_2, v_3)$. Hence there are at most $961 \cdot 19^{4|S|}$ possibilities for $(\mathbf{u}, \mathbf{v})$.

Finally we will treat the case $\alpha_2 = 0$ and $\alpha_1, \alpha_3, \lambda_1, \lambda_2$ still fixed and non-zero. We remark that the remaining two cases $\lambda_1 = 0$ and $\alpha_1 = 0$ can be dealt with using the same argument as the case $\lambda_2 = 0$ and $\alpha_2 = 0$ respectively. Note that $u_2 = \lambda_2 \in \mathbb{F}_p^\ast$. Using $\lambda_1 = u_1 - \frac{\alpha_1}{\alpha_3}u_3$ and $\mathbf{u} \not \in (K^p)^3$, we deduce that $u_1, u_3 \not \in K^p$. Hence the unit equation
\[
\lambda_1 = u_1 - \frac{\alpha_1}{\alpha_3}u_3
\]
has at most $31 \cdot 19^{2|S|}$ solutions $(u_1, u_3)$. Similarly, the unit equation
\[
\lambda_1 v_1 + \lambda_2 v_2 = 1
\]
has at most $31 \cdot 19^{2|S|}$ solutions $(v_1, v_2)$. Since $v_1$ determines $v_3$, we have proven that there are also at most $961 \cdot 19^{4|S|}$ possibilities for $(\mathbf{u}, \mathbf{v})$ in this case.

So far we have treated $\alpha_1, \alpha_2, \alpha_3, \lambda_1, \lambda_2$ as fixed. To every element of $B_1'$ we can attach a tuple $\mathbf{t} = (\alpha_1, \alpha_2, \alpha_3, \lambda_1, \lambda_2)$. Clearly there are at most $p^5$ such tuples. Furthermore, we have shown that for each fixed tuple $\mathbf{t}$ there are at most $961 \cdot 19^{4|S|}$ $(\mathbf{u}, \mathbf{v}) \in B_1'$ that correspond to $\mathbf{t}$. Altogether we have proven that $|B_1'| \leq 961 \cdot p^5 \cdot 19^{4|S|}$.

To deal with $B_q$ one can use a very similar approach, so we will only sketch the proof. In this case we define
\begin{align*}
B_q' := \{(\mathbf{u}, \mathbf{v}) &\in (\mathcal{O}_S^\ast)^3 \times (\mathcal{O}_S^\ast)^3: \mathbf{u}, \mathbf{v} \not \in (K^q)^3, u_i \not \in \mathbb{F}_q^\ast \text{ or } v_i \not \in \mathbb{F}_q^\ast, H_K(u_i) \leq c_{K, S}, \\
H_K(v_i) &\leq \frac{q}{p}\left(\omega(S) + 2g - 2\right), u_1 v_1^{q^{f}} + u_2 v_2^{q^{f}} + u_3 v_3^{q^{f}} = 1 \text{ for all } f \in \mathbb{Z}_{\geq 0}\}.
\end{align*}
Note that we now only require that $\mathbf{u}, \mathbf{v} \not \in (K^q)^3$ instead of $\mathbf{u}, \mathbf{v} \not \in (K^p)^3$. In our new setting we find that $\alpha_1, \alpha_2, \alpha_3, \lambda_1, \lambda_2 \in \mathbb{F}_q$ instead of $\alpha_1, \alpha_2, \alpha_3, \lambda_1, \lambda_2 \in \mathbb{F}_p$. This means that we have $q^5$ tuples $(\alpha_1, \alpha_2, \alpha_3, \lambda_1, \lambda_2)$. For each fixed tuple $\mathbf{t}$ there are at most $\log_p(q) \cdot 961 \cdot 19^{4|S|}$ $(\mathbf{u}, \mathbf{v}) \in B_q'$ that can map to $\mathbf{t}$. The extra factor $\log_p(q)$ comes from the fact that we merely know that $\mathbf{u}, \mathbf{v} \not \in (K^q)^3$ when we apply Theorem \ref{tKP}. We conclude that $|B_q'| \leq 961 \cdot \log_p(q) \cdot q^5 \cdot 19^{4|S|}$.

Our only remaining task is to bound $|A|$. We start by recalling a ``gap principle''. Define
\begin{align*}
\mathcal{S} &:= \{(x_0 : x_1 : x_2 : x_3) \in \mathbb{P}^3(K) \setminus \mathbb{P}^3(\mathbb{F}_q) : x_0 + x_1 + x_2 = x_3, \\
&v(x_0) = v(x_1) = v(x_2) = v(x_3) \text{ for every } v \in M_K \setminus S\}.
\end{align*}
Then we have the following lemma.

\begin{lemma}[Gap principle]
\label{lGap}
Let $B$ be a real number with $\frac{3}{4} < B < 1$, and let $P > 0$. Then the set of projective points $(x_0 : x_1 : x_2 : x_3)$ of $\mathcal{S}$ with
\[
P \leq H_K(x_0 : x_1 : x_2 : x_3) < \left(1 + \frac{4B - 3}{2}\right)P
\]
is contained in the union of at most $4^{|S|} (e/(1 - B))^{3|S| - 1}$ $1$-dimensional projective subspaces of $x_0 + x_1 + x_2 = x_3$.
\end{lemma}

\begin{proof}
This was proved in \cite{EG} for function fields in characteristic $0$, but the proof works ad verbatim in characteristic $p$.
\end{proof}

Take any $P > 0$ and suppose that $(x, y, z) \in A$ is a solution to
\[
x + y + z = 1
\]
with $P \leq H_K(x : y : z : 1) < \left(1 + \frac{4B - 3}{2}\right)P$. Then we can apply Lemma \ref{lGap} to deduce that $(x : y : z : 1)$ is contained in some $1$-dimensional projective subspace. This means that $x, y, z$ satisfy an additional equation
\[
ax + by + cz = d
\]
for some $a, b, c, d \in K$, such that the equation is independent from the equation $x + y + z = 1$. We may assume without loss of generality that $a \neq 0$. This implies
\begin{align}
\label{eNew}
(a - b)y + (a - c)z = a - d.
\end{align}
If $a - b$, $a - c$ and $a - d$ are zero, we conclude that $a = b = c = d$. This is a contradiction, since we assumed that the equation $ax + by + cz = d$ was linearly independent from the equation $x + y + z = 1$. If only one of $a - b$, $a - c$ and $a - d$ is not zero, we find that $y = 0$, $z = 0$ and $0 = a - d \neq 0$ respectively, so we obtain a contradiction in every case. From now on we will assume that $a - b \neq 0$ and distinguish three cases.

\vspace{\baselineskip}

\noindent \textbf{Case I: $a - c \neq 0$, $a - d \neq 0$.}
In this case we view (\ref{eNew}) as a unit equation. Since $(x, y, z) \in A$, it follows that $H_K(x), H_K(y), H_K(z) \leq c'_{K, S}$. We conclude that
\[
H_K((a - b)y) \in [H_K(a - b) - c'_{K, S}, H_K(a - b) + c'_{K, S}].
\]
Theorem \ref{tKP} implies that there are at most $q^2 + (\log_p(2c'_{K, S}) + 1) \cdot 31 \cdot 19^{2|S|}$ solutions $(y, z)$ to (\ref{eNew}). From $x + y + z = 1$ we see that $y$ and $z$ determine $x$. 

We will now count the total contribution to the number of solutions from case I. Choose $B := \frac{7}{8}$. Note that
\[
H_K(x : y : z : 1) \leq H_K(x) + H_K(y) + H_K(z) \leq 3c'_{K, S}.
\]
Now define $l := \log_{\frac{5}{4}}(3c'_{K, S}) + 1$. Then for every solution $(x, y, z) \in A$ there is $i$ with $0 \leq i < l$ such that
\[
\left(\frac{5}{4}\right)^i \leq H_K(x : y : z : 1) < \left(\frac{5}{4}\right)^{i + 1}.
\]
For fixed $i$ every solution $(x : y : z : 1)$ is contained in the union of at most $(2048e^3)^{|S|}$ $1$-dimensional projective subspaces. Furthermore, we have just shown that each subspace contains at most $q^2 + (\log_p(2c'_{K, S}) + 1) \cdot 31 \cdot 19^{2|S|}$ solutions. This gives as total bound for $A$ in case I
\begin{align}
|A| &\leq (\log_{\frac{5}{4}}(3c'_{K, S}) + 1) \cdot (2048e^3)^{|S|} \cdot q^2 \cdot (\log_p(2c'_{K, S}) + 1) \cdot 31 \cdot 19^{2|S|} \nonumber \\
&\leq 31q^2 \cdot (\log_{\frac{5}{4}}(3c'_{K, S}) + 1)^2 \cdot (15 \cdot 10^6)^{|S|}. \label{eAI}
\end{align}

\noindent \textbf{Case II: $a - c \neq 0$, $a - d = 0$.}
In this case (\ref{eNew}) gives
\[
z = -\frac{a - b}{a - c} y.
\]
Substitution in $x + y + z = 1$ yields
\begin{align}
\label{eNew2}
x + \left(1 - \frac{a - b}{a - c}\right) y = 1.
\end{align}
If $a - b = a - c$, we see that $x = 1$, contrary to our assumption $x \not \in \mathbb{F}_q$. So we will assume that $a - b \neq a - c$ and treat (\ref{eNew2}) as a unit equation. Then, following the proof of case I, we get the bound (\ref{eAI}) for $A$ in case II.

\vspace{\baselineskip}

\noindent \textbf{Case III: $a - c = 0$, $a - d \neq 0$.}
From (\ref{eNew}) we deduce that
\[
y = \frac{a - d}{a - b}.
\]
If $a - b = a - d$, we conclude that $y = 1$, which is again a contradiction. Substitution in $x + y + z = 1$ gives
\begin{align}
\label{eNew3}
x + z = 1 - \frac{a - d}{a - b}.
\end{align}
Note that (\ref{eNew3}) is another unit equation and, just as before, we obtain the bound (\ref{eAI}) for $A$ in case III.

\section{Application to Fermat surfaces}
The goal of this section is to prove Theorem \ref{tFermat}. We start off with a definition.

\begin{definition}
We say that a valuation $v$ of $K$ is $D$-generic if the following two conditions are satisfied
\begin{itemize}
\item first of all 
\[
v\left(\frac{Dx}{x}\right) = -1
\]
for all $x \in K^\ast$ satisfying $p \nmid v(x)$;
\item and secondly
\[
v\left(\frac{Dx}{x}\right) \geq 0
\]
for all $x \in K^\ast$ with $p \mid v(x)$.
\end{itemize}
In $\mathbb{F}_p(t)$ and $D$ differentiation with respect to $t$, every valuation is $D$-generic except for the infinite valuation. In general only finitely many valuations are not generic.
\end{definition}

In this section $K$ and $D$ will always be equal to respectively $\mathbb{F}_p(t)$ and differentiation with respect to $t$. Whenever we say that $v$ is generic, we will mean generic with respect to this $D$. Let $N$ be a $(480, p)$-good integer coprime to $p$. Suppose that $x, y, z \in \mathbb{F}_p(t)$ is a solution to
\[
x^N + y^N + z^N = 1
\]
satisfying the conditions of Theorem \ref{tFermat}, i.e. $x, y, z \not \in \mathbb{F}_p(t^p)$, $\frac{x}{Dx} \neq \frac{y}{Dy}$, $\frac{x}{Dx} \neq \frac{z}{Dz}$, $\frac{y}{Dy} \neq \frac{z}{Dz}$. Then differentiation with respect to $D$ yields
\begin{align*}
x^N\left(1 - \frac{z}{Dz} \frac{Dx}{x} \right) + y^N\left(1 - \frac{z}{Dz} \frac{Dy}{y} \right) &= 1 \\
x^N\left(1 - \frac{y}{Dy} \frac{Dx}{x} \right) + z^N\left(1 - \frac{y}{Dy} \frac{Dz}{z} \right) &= 1.
\end{align*}
Define
\[
S := \{v \in M_K : v(x) \neq 0 \text{ or } v(y) \neq 0 \text{ or } v(z) \neq 0\}.
\]
We may assume that $x$ is such that $\omega(x) \geq \frac{\omega(S)}{3}$. If $N > 12$, thanks to Lemma \ref{lHeight}, we have
\[
H_K(x^N) = NH_K(x) > 6\omega(x) \geq 2\omega(S) \geq H_K\left(1 - \frac{z}{Dz} \frac{Dx}{x} \right)
\]
and similarly 
\[
H_K(x^N) >  H_K\left(1 - \frac{y}{Dy} \frac{Dx}{x} \right).
\]
 Hence $x^N\left(1 - \frac{z}{Dz} \frac{Dx}{x} \right),x^N\left(1 - \frac{y}{Dy} \frac{Dx}{x} \right) \not \in \mathbb{F}_p$ and therefore we can write
\begin{align*}
x^N \left(1 - \frac{z}{Dz} \frac{Dx}{x} \right) &= \delta^{p^s} \\
x^N \left(1 - \frac{y}{Dy} \frac{Dx}{x} \right) &= \epsilon^{p^r}
\end{align*}
with $\delta, \epsilon \not \in \mathbb{F}_p(t^p)$. Now we claim that for $N > 48$
\begin{align}
\label{eGamma}
\omega(\delta) \geq \frac{\omega(S)}{4}.
\end{align}
Indeed suppose for the sake of contradiction that $\omega(\delta) < \frac{\omega(S)}{4}$. Then there is a finite subset $T$ of $M_K$ with $\omega(T) \geq \frac{\omega(S)}{12}$ such that for all $v \in T$ we have $v(x) \neq 0$ and $v(\delta) = 0$. For such a valuation $v \in T$ we have
\[
N \mid v\left(1 - \frac{z}{Dz} \frac{Dx}{x} \right).
\]
This implies that
\[
4\omega(S) \geq 2H_K\left(1 - \frac{z}{Dz} \frac{Dx}{x} \right) \geq \sum_{v \in T} \left|v\left(1 - \frac{z}{Dz} \frac{Dx}{x} \right)\right| \text{deg}(v) \geq N \frac{\omega(S)}{12}.
\]
This is impossible for $N > 48$, so we have established (\ref{eGamma}). For convenience we define for a valuation $v$ and $a, b \not \in \mathbb{F}_p(t^p)$
\begin{align*}
f_v(a, b) &:= \left|v\left(1 - \frac{a}{Da} \frac{Db}{b} \right) \right|, \\
g_v(x, y, z) := |v(\delta)| + |v(\epsilon)| + f_v(x, y) &+ f_v(y, x) + f_v(x, z) + f_v(z, x) + f_v(y, z) + f_v(z, y).
\end{align*}
Our next claim is that there is a generic place $v \in M_K$ such that $v(\delta) \neq 0$ and
\begin{align}
\label{eSv}
g_v(x, y, z) \leq 480.
\end{align}
Indeed, Lemma \ref{lHeight} and Theorem \ref{tMason} give the following bound
\[
\sum_{\substack{v \in M_K \\ v(\delta) \neq 0}} g_v(x, y, z) \deg(v) \leq 60 \omega(S).
\]
Note that there are at least two places such that $v(\delta) \neq 0$, so there is at least one generic place $v$ such that $v(\delta) \neq 0$. Hence if $\omega(S) \leq 8$, (\ref{eSv}) follows immediately. So suppose that $\omega(S) > 8$. Using (\ref{eGamma}) we conclude that
\[
\left(\frac{\omega(S)}{4} - 1\right) \min_{\substack{v \in M_K \\ v(\delta) \neq 0 \\ v \text{ generic}}} g_v(x, y, z)  \leq (\omega(\delta) - 1) \min_{\substack{v \in M_K \\ v(\delta) \neq 0 \\ v \text{ generic}}} g_v(x, y, z) \leq 60 \omega(S).
\]
In this case (\ref{eSv}) follows from our assumption $\omega(S) > 8$, completing the proof of our claim. From now on fix a generic $v \in M_K$ satisfying $v(\delta) \neq 0$ and (\ref{eSv}). Note that
\begin{align}
\label{eRel}
v\left(1 - \frac{z}{Dz} \frac{Dx}{x} \right) + Nv(x) = p^s v(\delta).
\end{align}
Clearly we may assume that $s > 0$ and $r > 0$, otherwise we can directly apply Theorem \ref{tMason}. Hence if $p > 480$, we find that $v(x) \neq 0$. If furthermore $N > 480$, we also find that $v\left(1 - \frac{z}{Dz} \frac{Dx}{x} \right) \neq 0$. Finally observe that
\[
N \mid p^s v(\delta) - v\left(1 - \frac{z}{Dz} \frac{Dx}{x} \right).
\]
We now distinguish two cases. First suppose that $v(\delta) > 0$. Then clearly also $v(x) > 0$. If furthermore $v\left(1 - \frac{z}{Dz} \frac{Dx}{x} \right) < 0$, we get that $N$ divides $a p^s + b$ with $0 < a, b \leq 480$ contrary to our assumptions. So from now on we assume that 
\begin{align}
\label{ePos}
v\left(1 - \frac{z}{Dz} \frac{Dx}{x} \right) > 0.
\end{align}
Now comes the crucial observation that $p \nmid v(x)$. Indeed, otherwise we find by (\ref{eRel})
\[
p \mid v\left(1 - \frac{z}{Dz} \frac{Dx}{x} \right),
\]
which is not possible due to $p > 480$, (\ref{eSv}) and (\ref{ePos}). Hence we deduce for a generic valuation $v$ that $v\left(\frac{Dx}{x}\right) = -1$. Combining this with (\ref{ePos}) again we get that $v(z) \neq 0$. Just as in (\ref{eRel}) we have
\[
v\left(1 - \frac{y}{Dy} \frac{Dx}{x} \right) + Nv(x) = p^r v(\epsilon).
\]
Recall that $v(x) > 0$, hence $v(\epsilon) > 0$. But this gives
\[
v\left(1 - \frac{y}{Dy} \frac{Dz}{z} \right) + Nv(z) = 0,
\]
which is a contradiction for $N > 480$.

We still need to treat the case $v(\delta) < 0$. In that case we find that $v(x) < 0$ and $v\left(1 - \frac{z}{Dz} \frac{Dx}{x} \right) < 0$. Similarly as before we can show that this implies $p \mid v(z)$ for a generic valuation $v$. Note that
\[
z^N\left(1 - \frac{y}{Dy} \frac{Dz}{z} \right) = (1 - \epsilon)^{p^r}.
\]
Since $v(x) < 0$ implies that $v(\epsilon) < 0$, we find that
\begin{align}
\label{eRel2}
v\left(1 - \frac{y}{Dy} \frac{Dz}{z} \right) + Nv(z) = p^r v(1 - \epsilon) = p^r v(\epsilon).
\end{align}
Combining (\ref{eRel2}) with $p \mid v(z)$ we get that
\[
p \mid v\left(1 - \frac{y}{Dy} \frac{Dz}{z} \right).
\]
If $p > 480$, then (\ref{eSv}) implies that $v\left(1 - \frac{y}{Dy} \frac{Dz}{z} \right) = 0$. Hence (\ref{eRel2}) gives $N \mid v(\epsilon)$. Using (\ref{eSv}) and $N > 480$ once more we conclude that $v(\epsilon) = 0$, which is the desired contradiction.

\section{Curves inside Fermat surfaces}
\label{sCounter}
The goal of this section is to show that Theorem \ref{tFermat} becomes false if we allow $x$, $y$, $z$, $x/y$, $x/z$ or $y/z$ to be in $\mathbb{F}_p(t^p)$. By symmetry it suffices to do this in the case $x$ or $y/z$ in $\mathbb{F}_p(t^p)$. We will do this by exhibiting explicit curves inside the Fermat surface.

Let us start by allowing $y/z \in \mathbb{F}_p(t^p)$. We can rewrite
\[
x^N + y^N + z^N = 1
\]
as
\[
\frac{1}{1 - x^N} y^N + \frac{1}{1 - x^N} z^N = 1.
\]
Then if $N$ is odd, we have
\[
\frac{1}{1 - x^N} y^N + \frac{-x^N}{1 - x^N} \frac{(-z)^N}{x^N} = 1.
\]
The key point is that we can now put $\alpha := \frac{1}{1 - x^N}$, $\tilde{z} = \frac{-z}{x}$, after which the last equation can be rewritten as
\begin{align}
\label{eCat2}
\alpha y^N + (1 - \alpha) \tilde{z}^N = 1. 
\end{align}
But it is rather straightforward to find solutions to this last equation. Indeed, we know that $N \mid p^k - 1$ for some $k > 0$. For such a $k$ we put
\[
y := \alpha^{\frac{p^k - 1}{N}}, \tilde{z} := (1 - \alpha)^{\frac{p^k - 1}{N}},
\]
and one easily verifies that $y$ and $\tilde{z}$ satisfy (\ref{eCat2}). Going back to our original variables $x$, $y$ and $z$ we get that
\[
y := \left(\frac{1}{1 - x^N}\right)^{\frac{p^k - 1}{N}}, z := -x \left(\frac{-x^N}{1 - x^N}\right)^{\frac{p^k - 1}{N}}.
\]
There are two important remarks to make about the above construction. First of all, it is easily verified that $y/z \in \mathbb{F}_p(t^p)$ as we claimed. Secondly, we used that $N$ is odd during our construction. However, we only need that $-1$ is an $N$-th power in $\mathbb{F}_p^\ast$.

Now suppose that $x \in \mathbb{F}_p(t^p)$. For simplicity we will again assume that $N$ is odd. Then from the equation
\[
x^N + y^N + z^N = 1
\]
we find that
\[
\left(\frac{1}{z}\right)^N + \left(\frac{-x}{z}\right)^N + \left(\frac{-y}{z}\right)^N = 1.
\]
After putting $\tilde{x} = \frac{1}{z}$, $\tilde{y} = \frac{-x}{z}$ and $\tilde{z} = \frac{-y}{z}$ we get that
\[
\tilde{x}^N + \tilde{y}^N + \tilde{z}^N = 1
\]
with $\frac{\tilde{y}}{\tilde{z}} \in \mathbb{F}_p(t^p)$. Hence we can apply the previous construction.

Finally we will explain why we need the condition that $N$ is $(480, p)$-good. If $N = p^r + 1$ for some $r \geq 0$, it is possible to write down non-trivial lines on the Fermat surface, see Section 5.1-5.4 of \cite{Ronald}. It turns out that our method is unable to distinguish between the case $N = p^r + 1$ and $N = ap^r + b$ with $0 < a, b$ small. This may seem strange at first, but it is in fact quite natural.

Indeed, let us compare this with the situation in characteristic $0$. In this case it follows from the work of Voloch \cite{Voloch2} that for $N$ sufficiently large the equation
\[
x^N + y^N + z^N = 1
\]
has no non-constant solutions $x, y, z \in \mathbb{C}(t)$. In fact, this is a rather easy consequence from his abc Theorem. However, it is a more difficult task to find the smallest $N$ using abc Theorems, see for example \cite{CZ}. Our Theorem \ref{tFermat} is also based on abc type arguments and for this reason it should not be surprising that we can not distinguish between the case $N = p^r + 1$, giving unirational surfaces \cite{Ronald}, and $N = ap^r + b$ with $0 < a, b$ small.

Thus, morally, the notion of $N$ being $(480,p)$-good in Theorem \ref{eFermat} can be interpreted as saying that $N$ is ``far enough" from an exponent that gives a unirational surface. In the proof we use this condition when we analyze the $2$-Frobenius families. It is therefore instructive to notice here that there is a partial converse. Namely, we can use the description given at the beginning of Section \ref{2-Frob} to produce non-trivial rational curves on Fermat surfaces. We will assume $p \equiv 1 \mod 4$ for simplicity: a similar computation can be carried out for the case $p \equiv 3 \mod 4$. 

We will use the notation of Section \ref{2-Frob}. Rename $\tilde{\alpha_1}=\frac{\alpha_1}{\alpha_3}$ and $\tilde{\alpha_2}=\frac{\alpha_2}{\alpha_3}$. Choose $\tilde{\alpha_1},\tilde{\alpha_2} \neq 0$ such that 
\[
\tilde{\alpha_1}^2 + \tilde{\alpha_2}^2 = -1
\]
and put $\lambda_1=i\tilde{\alpha_2}$ and $\lambda_2=i\tilde{\alpha_1}$, where $i$ is an element of $\mathbb{F}_p$ such that $i^2 = -1$. We further impose the conditions
\[
u_1 = v_1, u_2 = v_2, u_3 = v_3.
\]
With these choices, one can check that all the relevant equations in Section \ref{2-Frob} are satisfied for $(v_1, v_2, v_3) = (\tilde{\alpha_1}t + i\tilde{\alpha_2}, \tilde{\alpha_2}t + i\tilde{\alpha_1}, t)$. Thus, since all the implications at the beginning of \ref{2-Frob} are reversible, one deduces that the line $(\tilde{\alpha_1}t + i\tilde{\alpha_2}, \tilde{\alpha_2}t + i\tilde{\alpha_1}, t)$ is contained in \emph{all} Fermat surfaces $x^{p^s + 1} + y^{p^s + 1} + z^{p^s + 1} = 1$. Alternatively, one may directly verify that this yields lines on Fermat surfaces.

We conclude by remarking that the height bound in Theorem \ref{eHeight} \emph{can not} be improved to a \emph{linear} height bound in $\omega(S)$. Indeed, this follows easily by using the curves we constructed at the beginning of this section. A natural question is whether the quadratic dependency on $\omega(S)$ is sharp.

\section{Acknowledgements}
We thank Jan-Hendrik Evertse for giving us this problem, useful discussions and proofreading. We would also like to thank Hendrik Lenstra and Ronald van Luijk for useful discussions.

\end{document}